\documentclass[a4paper,12pt]{amsart}

\usepackage{amsmath,amsthm}
\usepackage{mathtools}
\usepackage{color}
\usepackage{enumitem}
\usepackage[colorlinks,citecolor=blue,urlcolor=blue]{hyperref}

\usepackage{tikz}
\usetikzlibrary{positioning}

\newtheorem{theorem}{Theorem}

\newtheorem{lemma}{Lemma}
\theoremstyle{remark}
\newtheorem*{remark}{Remark}

\newcommand{\E}{\mathbb{E}}
\newcommand{\eye}{{E_{\rm a}}}
\newcommand{\RV}{Z}

\newcommand{\reals}{\mathbb{R}}
\newcommand{\complexes}{\mathbb{C}}
\newcommand{\integers}{\mathbb{Z}}
\newcommand{\naturalnumbers}{\mathbb{N}}

\DeclareMathOperator{\supp}{supp}
\DeclareMathOperator{\Span}{span}
\DeclareMathOperator{\interior}{int}

\title{On the regularity of complex multiplicative chaos}

\author[J. Junnila]{Janne Junnila}
\address{MATHAA Institute, \'Ecole Polytechnique F\'ed\'erale de Lausanne, Station 8, 1015 Lausanne, Switzerland}
 \email{janne.junnila@epfl.ch}
 \author[E. Saksman]{Eero Saksman}
\address{University of Helsinki, Department of Mathematics and Statistics, Finland}
\email{eero.saksman@helsinki.fi}
\author[L. Viitasaari]{Lauri Viitasaari}
\address{University of Helsinki, Department of Mathematics and Statistics, Finland}
\email{lauri.viitasaari@iki.fi}

\begin{document}

\renewcommand{\Re}{\operatorname{Re}}
\renewcommand{\Im}{\operatorname{Im}}

\begin{abstract}
Denote by $\mu_\beta="\exp(\beta X)"$ the Gaussian multiplicative chaos
which is defined using a log-correlated Gaussian field $X$ on a domain
$U\subset\reals^d$. The case $\beta\in\reals$ has been studied quite
intensively, and then $\mu_\beta$ is a random measure on $U$. It is
known that $\mu_\beta$ can also be defined for complex values $\beta$
lying in certain subdomain of $\complexes$, and then the realizations of
$\mu_\beta$ are random generalized functions on $U$. In this note we
complement the results of \cite{JSW1} (where the case of purely
imaginary $\beta$ was considered) by studying the Besov-regularity of
$\mu_\beta$ and the finiteness of moments for general complex values of
$\beta$.\end{abstract}

\maketitle

\section{Introduction}

Gaussian multiplicative chaos measures on a domain $U \subset \reals^d$ are positive measures which can be formally written as
\begin{equation}\label{eq:gmc}
  \mu_\beta(dx) = e^{\beta X(x) - \frac{\beta^2}{2} \E X(x)^2} \, dx = \;\; :e^{\beta X(x)}: dx,
\end{equation}
where $\beta \in (0,\sqrt{2d})$ is a parameter and $X$ is a log-correlated Gaussian field. Here $:e^{\beta X(x)}:$ refers to the Wick exponential of the field $X$. By a log-correlated Gaussian field we mean that $X$ is a centered Gaussian distribution (generalized function) with covariance kernel of the form
\begin{equation}\label{eq:cov}
  \E X(x) X(y) = \log \frac{1}{|x-y|} + g(x,y),
\end{equation}
where $g$ is some sufficiently regular function.

GMC measures first appeared around the same time in the early 70s in two rather different contexts, first by Hoegh-Krohn \cite{Hoegh-Krohn} and then by Mandelbrot \cite{Mandelbrot}, but it was only in 1985 when Kahane built a rigorous theory of Gaussian multiplicative chaos \cite{Kahane}. As the point evaluations of log-correlated fields are not well-defined, giving a rigorous meaning to \eqref{eq:gmc} is usually done by approximating the field $X$ with some regular fields $X_\varepsilon$ and showing that the measures $e^{\beta X_\varepsilon(x) - \frac{\beta^2}{2} \E X_\varepsilon(x)^2} \, dx$ converge as $\varepsilon \to 0$. See \cite{Berestycki} for an elegant proof of existence of non-trivial measures in the case of convolution approximations and \cite{RV2} for a review of GMC measures.

Today GMC measures are encountered in various settings such as random geometry and Liouville quantum gravity \cite{AJKS,DKRV,DS,KRV,Sheffield}, random matrices \cite{BWW,Webb} and number theory \cite{SW}. In this note we are interested in the remarkable fact that $\mu_\beta$ may also be defined for certain complex values of $\beta$. For non-real values of $\beta$ the resulting object is in general not a measure but a random generalized function. It turns out that $\mu_\beta$ is well-defined and analytic in $\beta$ in the eye-shaped open domain
\[\eye = \interior \Span(\pm \sqrt{2d} \cup B(0,\sqrt{d})),\]
see Figure~\ref{fig:eye}.
This was originally observed for cascade analogues in \cite{Barral} and extended to some special log-correlated fields $X$ in \cite{AJKS}. Recently analyticity was extended to a general class of log-correlated fields in \cite{JSW2}. A related but different kind of complex chaos was studied in \cite{LRV}.

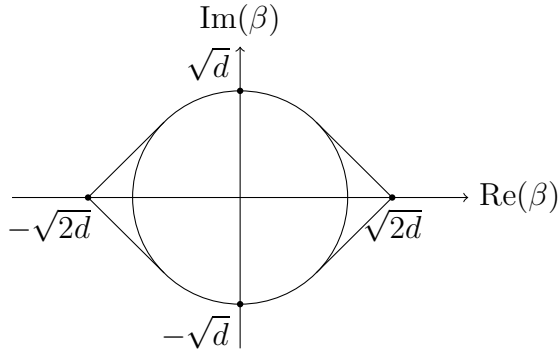
\begin{figure}
  \begin{tikzpicture}[x=1.0cm,y=1.0cm]
    \def\sqrttwo{1.4142135623730951}

    \draw[->] (-3,0) -- (3,0) node[right] {$\Re(\beta)$};
    \draw[->] (0,-2) -- (0,2) node[above] {$\Im(\beta)$};
    
    \draw (0,0) circle (\sqrttwo);
    \draw (1,1) -- (2,0) -- (1,-1);
    \draw (-1,1) -- (-2,0) -- (-1,-1);
    \fill (0,\sqrttwo) circle (1.2pt) node[anchor=south east] {$\sqrt{d}$};
    \fill (0,-\sqrttwo) circle (1.2pt) node[anchor=north east] {$-\sqrt{d}$};
    \fill (-2,0) circle (1.2pt) node[below left = 0pt and -5pt] {$-\sqrt{2d}$};
    \fill (2,0) circle (1.2pt) node[anchor=north] {$\sqrt{2d}$};
  \end{tikzpicture}
  \caption{The subcritical regime $\eye$ for $\beta$ in the complex plane.}\label{fig:eye}
\end{figure}

In a recent paper \cite{JSW1} it was rigorously verified that the scaling limit of the spin field of the XOR-Ising model is given by the real part of an imaginary multiplicative chaos distribution corresponding to the parameter value $\beta = i/\sqrt{2}$. In the same paper there was a detailed study of the regularity of these imaginary chaos distributions, some of the main results of which are stated in the next theorem.

\begin{theorem}[\cite{JSW1}]
  Assume that the domain $U$ is bounded and simply connected and that the function $g$ in the covariance \eqref{eq:cov} is continuous, integrable and bounded from above. Let $\beta \in i(0,\sqrt{d})$. Then
  \begin{enumerate}[label={\rm (\roman*)}]
    \item We have $\E |\mu(f)|^p < \infty$ for any $f \in C_c^\infty(U)$ and all $p \ge 1$.
    \item $\mu$ is almost surely not a complex measure.
    \item We have almost surely $\mu \in B_{p,q,loc}^s(U)$ when $s < -\frac{|\beta|^2}{2}$ and $\mu\notin B_{p,q,loc}^s(U)$ when $s > -\frac{|\beta|^2}{2}$.
\end{enumerate}
\end{theorem}

We refer to \cite[Section~2.2]{JSW1} and the references there for basic facts on function spaces, especially on Besov spaces.

In the case of moments, there is a rather striking difference between this theorem and the following corresponding result for real $\beta$.

\begin{theorem}[\cite{Kahane,Molchan,RobertVargas}]
  Let $\beta \in (0,\sqrt{2d})$ and $p \in \reals$. Then for any $f \in C_c^\infty(U)$ we have $\E |\mu(f)|^p < \infty$ if and only if $p < \frac{2d}{\beta^2}$.
\end{theorem}

The goal of this paper is to prove analogous positive results for general complex $\beta$.  In \cite{JSW2} the complex  chaos was constructed as $H^{-d}_{loc}(U)$-valued random analytic function, but the regularity was not studied in detail. 

For the existence of the moments we have the following theorem.

\begin{theorem}\label{thm:moments}
  Assume that in \eqref{eq:cov} we have $g \in H_{loc}^{d + \varepsilon}(U \times U)$. Let $f \in C_c^\infty(B)$. Then
  for a given $p \ge 1$ there exists a constant $C > 0$ such that we have $\E |\mu_\beta(f)|^p \le C (\|f\|_\infty)^p$ for all $\beta \in \eye_p$, where
  \[\eye_p \coloneqq \eye \cap \Big(\big\{|\Re(\beta)| < \frac{\sqrt{2d}}{p}\big\} \cup \big\{(p-1)\Re(\beta)^2 + \Im(\beta)^2 < \frac{2d(p-1)}{p}\big\}\Big).\]
  See Figure~\ref{fig:eyep} for an illustration of the region $\eye_p$.
\end{theorem}

\begin{figure}
  \begin{tikzpicture}[x=1.0cm,y=1.0cm]
    \def\sqrttwo{1.4142135623730951}
    \fill[white] (-5,-2) -- (8,-2) -- (8,2) -- (-5,2) -- cycle;

    \fill[lightgray] (1.5,-0.5) arc[x radius=1.7320508075688772, y radius=1.0, start angle=-30, end angle=30]
    -- (1.0,1.0) arc[start angle=45, end angle=135, radius=\sqrttwo]
    -- (-1.5,0.5) arc[x radius=1.7320508075688772, y radius=1.0, start angle=150, end angle=210]
    -- (-1.0,-1.0) arc[start angle=225, end angle=315, radius=\sqrttwo];

    \draw[->] (-3,0) -- (3,0) node[right] {$\Re(\beta)$};
    \draw[->] (0,-2) -- (0,2) node[above] {$\Im(\beta)$};
    
    \draw (1,1) arc[start angle=45, end angle=135, radius=\sqrttwo];
    \draw (-1,-1) arc[start angle=225, end angle=315, radius=\sqrttwo];
    \draw (1,1) -- (2,0) -- (1,-1);
    \draw (-1,1) -- (-2,0) -- (-1,-1);
    \fill (0,\sqrttwo) circle (0.7pt);
    \fill (0,-\sqrttwo) circle (0.7pt);
    \fill (-2,0) circle (0.7pt);
    \fill (2,0) circle (0.7pt);
    \fill (1,1) circle (0.7pt);
    \fill (-1,1) circle (0.7pt);
    \fill (1,-1) circle (0.7pt);
    \fill (-1,-1) circle (0.7pt);
    \fill (1.5,0.5) circle (1.2pt) node[anchor=south west] {$(\sqrt{2d}/p, \sqrt{2d}(p-1)/p)$};
    \fill (1.5,-0.5) circle (1.2pt);
    \fill (-1.5,-0.5) circle (1.2pt);
    \fill (-1.5,0.5) circle (1.2pt);

    \draw[dotted] (-1,1) arc[start angle=135, end angle=225, radius=\sqrttwo];
    \draw[dotted] (1,-1) arc[start angle=-45, end angle=45, radius=\sqrttwo];
    \draw (1.5,-0.5) arc[x radius=1.7320508075688772, y radius=1.0, start angle=-30, end angle=30];
    \draw[dashed] (1.5,0.5) arc[x radius=1.7320508075688772, y radius=1.0, start angle=30, end angle=150];
    \draw (-1.5,0.5) arc[x radius=1.7320508075688772, y radius=1.0, start angle=150, end angle=210];
    \draw[dashed] (-1.5,-0.5) arc[x radius=1.7320508075688772, y radius=1.0, start angle=210, end angle=330];
  \end{tikzpicture}

  \begin{tikzpicture}[x=1.0cm,y=1.0cm]
    \def\sqrttwo{1.4142135623730951}

    \fill[white] (-5,-2) -- (8,-2) -- (8,2) -- (-5,2) -- cycle;

    \fill[lightgray] (0.7071067811865476,1.224744871391589) arc[radius=\sqrttwo, start angle=60, end angle=120]
    arc[start angle=135, end angle=225, x radius=1.0, y radius=1.7320508075688772]
    arc[start angle=240, end angle=300, radius=\sqrttwo]
    arc[start angle=-45, end angle=45, x radius=1.0, y radius=1.7320508075688772];

    \draw[->] (-3,0) -- (3,0) node[right] {$\Re(\beta)$};
    \draw[->] (0,-2) -- (0,2) node[above] {$\Im(\beta)$};
    
    \draw (1,1) arc[start angle=45, end angle=135, radius=\sqrttwo];
    \draw (-1,-1) arc[start angle=225, end angle=315, radius=\sqrttwo];
    \draw (1,1) -- (2,0) -- (1,-1);
    \draw (-1,1) -- (-2,0) -- (-1,-1);
    \fill (0,\sqrttwo) circle (0.7pt);
    \fill (0,-\sqrttwo) circle (0.7pt);
    \fill (-2,0) circle (0.7pt);
    \fill (2,0) circle (0.7pt);
    \fill (1,1) circle (0.7pt);
    \fill (-1,1) circle (0.7pt);
    \fill (1,-1) circle (0.7pt);
    \fill (-1,-1) circle (0.7pt);

    \draw[dotted] (-1,1) arc[start angle=135, end angle=225, radius=\sqrttwo];
    \draw[dotted] (1,-1) arc[start angle=-45, end angle=45, radius=\sqrttwo];
    \draw[dashed] (0.7071067811865476,1.224744871391589) arc[x radius=1.0, y radius=1.7320508075688772, start angle=45, end angle=135];
    \draw (-0.7071067811865476,1.224744871391589) arc[x radius=1.0, y radius=1.7320508075688772, start angle=135, end angle=225];
    \draw[dashed] (-0.7071067811865476,-1.224744871391589) arc[x radius=1.0, y radius=1.7320508075688772, start angle=225, end angle=315];
    \draw (0.7071067811865476,-1.224744871391589) arc[x radius=1.0, y radius=1.7320508075688772, start angle=-45, end angle=45];
  \end{tikzpicture}
  
  \caption{$\eye_p$ is the shaded area. In the first picture we have $1 < p < 2$ and in the second one $p > 2$. The dashed line is the ellipse $(p-1)x^2 + y^2 = 2d(p-1)/p$ and the dotted circle corresponds to the $L^2$-regime.}\label{fig:eyep}
\end{figure}
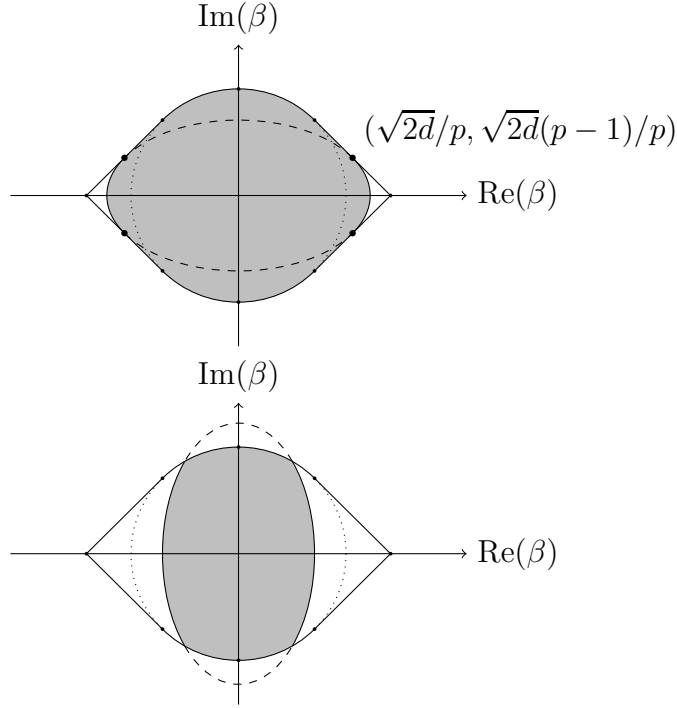

For the Besov-regularity we prove the following.

\begin{theorem}\label{thm:besov}
  Assume that in \eqref{eq:cov} we have $g \in C^\infty(U \times U)$. 
  Let $\beta \in \eye$ and $p \ge 1$. We have the following two cases:
  \begin{itemize}
  \item If $|\Re(\beta)| \le \frac{\sqrt{2d}}{p}$, then $\mu \in B_{p,q,loc}^s(U)$ for $s < - \frac{(p-1)\Re(\beta)^2 +\Im(\beta)^2}{2}$ and any $q \ge 1$.
  \item If $|\Re(\beta)| \ge \frac{\sqrt{2d}}{p}$, then $\mu \in B_{p,q,loc}^s(U)$ for $s < \frac{d}{p} - \sqrt{2d} |\Re(\beta)| + \frac{\Re(\beta)^2 - \Im(\beta)^2}{2}$ and any $q \ge 1$.
  \end{itemize}
\end{theorem}

\begin{remark}
  We suspect that the bounds obtained in Theorems \ref{thm:moments} and \ref{thm:besov} are optimal, but we do not touch this question in this paper. The smoothness condition on $g$ in Theorem~\ref{thm:besov} can be relaxed, see Remark~\ref{rmk:gregularity} below.
\end{remark}

\section{Auxiliary results and reduction to the exactly scaling field}

We refer to \cite{JSW2} for the construction of the complex chaos corresponding to the field $X$ with covariance \eqref{eq:cov} assuming that $g \in H^{d + \varepsilon}(U \times U)$. In particular, the existence implies that one may simply define $\mu_\beta$ for complex values of $\beta \in \eye$ as the analytic continuation of the standard chaos that corresponds to the parameter values $\beta \in \reals \cap \eye = (-\sqrt{2d}, \sqrt{2d})$.

A centred log-correlated field $X^e$ on a domain $U$ with the simple covariance structure  
\begin{equation}\label{eq:exact}
\E X^e(x)X^e(y)=\log (1/|x-y|),\qquad x,y\in U,
\end{equation}
is called \emph{exactly scaling}. Such fields exist locally in any dimension.

The chaos corresponding to $X^e$ has a very simple scaling property as Lemma~\ref{le:exact} shows and hence estimating the Besov norms via wavelets becomes easier. Our strategy is to reduce the case of a general field to this special case. We start with stating carefully the scaling relation. In what follows
$\langle \lambda, \phi\rangle$ stands for the standard duality pairing with the (test) function $\phi$ and the (generalised) function $\lambda$. Moreover, $\sim$ stands for equality in distribution.

\begin{lemma}\label{le:exact}
{\rm (i)}\quad In any dimension $d\geq 1$ there is $r_d>0$ so that $\log(1/|x-y|)$ is positive definite  on the ball $B(0,r_d)$.

\smallskip

\noindent{\rm (ii)}\quad Consider the complex chaos $\mu^e_{\beta}="e^{\beta X^e(x)}"$ for parameter values $\beta\in \eye$. For any $\varphi\in C_0^\infty (B(0,r_d))$ and $\varepsilon\in (0,1)$ it satisfies the scaling relation
\begin{equation}\label{eq:scaling}
\langle\mu^e_{\beta}, \varphi(\varepsilon^{-1}\cdot)\rangle\;\; {\sim}\;\; \varepsilon^d e^{\beta \RV-\beta^2\E \RV^2/2}\langle\mu^e_{\beta}, \varphi\rangle,
\end{equation}
where on the right hand side $\RV\sim N(0,\log(1/\varepsilon))$ is independent of $\mu^e_{\beta}$.
\end{lemma}
\begin{proof}
(i) This follows immediately from \cite[Theorem 4.5(i)]{JSW2}. 

(ii) In the case of real $\beta$ the scaling relation \eqref{eq:scaling} is well-known at least in dimension 1 to the experts (see \cite[Theorem 4]{BM}), but for the reader's convenience we give a full argument here. Directly from the covariance structure we have that for any $\varepsilon >0$ there is $\RV\sim N(0,\log(1/\varepsilon))$ so that
\begin{equation}\label{eq:scale}
X^e(\varepsilon \cdot)\sim X^e + \RV\qquad\textrm{on}\quad B(0,r_d).
\end{equation}
Consider convolution approximations $ X^e *\psi_\delta$,, where the bump function $\psi$ has a compact support, and $\psi_\delta:=\delta^{-d}\psi(\cdot/\delta)$. Denote $X^e_n :=X^e*\psi_{\varepsilon^n}$. For the standard approximation sequence $X^e_n$ of the field $X^e$ the above scaling relation takes the form
$$
X^e_{n+1}(\varepsilon\cdot)\sim  X^e_{n}+\RV,
$$
which holds in any fixed compact subset of $B(0,r_d)$ as soon as $n$ is large enough. Denote by $\mu^e_{\beta,n}(dx)= \exp\big(\beta X^e_{n+1}(x)-\beta^2\E(X^e_{n}(x))^2/2\big)dx$ the approximation of $\mu^e_\beta$ on level $n$ and observe that the above equality in distribution yields that
\begin{eqnarray*}
&&\langle\mu^e_{\beta,n+1},\varphi(\varepsilon^{-1}\cdot)\rangle\\
&=&\int_{B(0,r_d)}\exp\big(\beta X^e_{n+1}(x)-\beta^2\E(X^e_{n+1}(x))^2/2\big)\varphi(x/\varepsilon)dx\\
&\sim& \varepsilon^d e^{\beta \RV-\beta^2\E \RV^2/2}\int_{B(0,r_d)}\exp\big(\beta X^e_{n}(x)-\beta^2\E(X^e_{n}(x))^2/2\big)\varphi(x)dx\\
&=&\varepsilon^d e^{\beta \RV-\beta^2\E \RV^2/2}\langle\mu^e_{\beta,n},\varphi\rangle
\end{eqnarray*}
Letting $n\to\infty$ we obtain the claim for real $\beta$. Moreover, the same proof also shows that any $n$-tuple $(\mu^e_{\beta_1},\ldots, \mu^e_{\beta_n})$ satisfies a similar distributional scaling identity. Both sides of \eqref{eq:scaling} are random analytic functions in $\beta$ taking values in the Hilbert space $H^{-d}(B(0,r_d)$, and their restriction to $\eye\cap \reals$ has the same (joint in $\beta$) distribution, which easily yields by analytic continuation that they have the same distribution for all $\beta\in \eye$.
\end{proof}

Finally, we need the following result in order to pass from $X^e$ to a general  field $X$.
\begin{lemma}\label{lemma:reduction}
Let $\beta$ be real and $X$ be a log-correlated field on the ball $B\subset\reals^d$ with decomposition $X= X^e + G$ and covariance
\begin{equation}\label{eq:kova}
\E X(y)X(z) = \log \frac{1}{|y-z|} + g(y,z),
\end{equation}
where $g$ is $C^\infty$-smooth. Then $G$ is smooth.   
Denote by $\mu_\beta$ the chaos generated by the field $X$, and recall that $\mu^e_{\beta}$ stands for the chaos generated by $X^e$. Then for any test function $\psi\in C_c^\infty(U)$ and any $\beta\in \eye$ it holds that
\begin{equation}
\label{eq:formal-general-def}
\langle \mu_\beta \psi\rangle = \langle \mu^e_{\beta}, h_\beta \psi \rangle,
\end{equation}
where $h_\beta$ stands for the smooth random function
$$
h_\beta(x) := e^{\beta G(x)-\beta^2 g(x,x)/2}.
$$
In particular, since also $1/h_\beta$ is smooth, we deduce that the local Besov-smoothness of the chaos $X_\beta$ is the same as that of the chaos $X^e_{\beta}$.
\end{lemma}
\begin{proof}
It is enough to prove the claim for $\beta\in (0,\sqrt{2d})$ since the general claim then follows by analytic continuation.  We first observe that $G$ indeed is a Gaussian field with $C^\infty$-smooth realizations directly by the proof of Theorem A in \cite{JSW2}. Assume that $\beta\in (0,\sqrt{2d})$ and
consider the mollifications  $X_\delta, G_\delta$ and $X^e_{\delta}$ of the fields $X, G$ and $X^e$  with the same compactly supported radially symmetric mollifier. We then  have $X_\delta=X^e_{\delta}+G_\delta.$ Especially, since the covariance of a mollified field is obtained by mollifying the covariance separately with respect to  variable $x$ and variable $y$, the equality \eqref{eq:kova} yields that 
$$
\E X_\delta(x)X_\delta(x) =  \E X^e_{\delta}(x)X^e_{\delta}(x) + g_{\delta,\delta}(x,x),
$$
where $g_{\delta,\delta}$ stands for the mollification of $g$ with respect to both of the variables separately. We obtain
\begin{eqnarray*}
&&\exp\big(\beta X_\delta(x)-(\beta^2/2)\E X_\delta(x)^2\big)\\
&=&\exp\big(\beta G_\delta(x)-(\beta^2/2)g_{\delta,\delta}(x,x)\big) \exp\big(\beta X^e_{\delta}(x)-(\beta^2/2)\E X^e_{\delta}(x)^2\big),
\end{eqnarray*}
and the claim follows by observing that obviously, almost surely,
$$
\exp\big(\beta G_\delta(x)-(\beta^2/2)g_{\delta,\delta}(x,x)\big)\to\exp\big(\beta G(x)-(\beta^2/2)g(x,x)\big)
$$
locally in sup-norm as $\delta\to 0.$
\end{proof}

\begin{remark}\label{rmk:gregularity}
  The condition $g \in C^\infty(U \times U)$ can be replaced, for example, by the condition $g \in H_{loc}^{3d + \varepsilon}(U \times U)$. Namely, by \cite[Section~4.7.1]{RunstSickel} $B_{\infty,\infty,loc}^d(\reals^d)$ is locally a pointwise multiplier in all Besov spaces $B_{p,q,loc}^s(\reals^d)$ we use ($1 \le p,q \le \infty$ and $s \in (-d,0)$). It is easy to check that $e^{\beta G(x) - (\beta^2/2) g(x,x)} \in B_{\infty,\infty,loc}^d(\reals^d)$ under the condition above.
\end{remark}

\section{Proofs of Theorems \ref{thm:moments} and \ref{thm:besov}}

\begin{proof}[Proof of Theorem~\ref{thm:moments}]
  The case $1\leq p\leq 2$ essentially appears in the proof of \cite[Theorem~6.1]{JSW2}. There one writes $X$ as a sum of an almost $\star$-scale invariant field $L$ and an independent regular field $R$, obtaining approximating measures
  \[\mu_{n,\beta}(x) = \exp\big(\beta R(x) - \frac{\beta^2}{2} \E R(x)^2\big) \nu_{n,\beta}(x),\]
  with
  \[\nu_{n,\beta}(x) = \exp\big(\beta L_n(x) - \frac{\beta^2}{2} \E L_n(x)^2\big)\]
  and $(L_n)$ being a martingale approximation of the field $L$. We refer to the paper \cite{JSW2} (see especially proof of Theorem~6.1) for all the above notions. By independence, the regular part will not affect the finiteness of the moments, and thus the moment estimates obtained in  the rest of the proof apply.

  For the case $p > 2$ we simply note that one can replace the von Bahr--Esseen inequality used in the proof of \cite[Theorem~6.1]{JSW2}  by Rosenthal's inequality \cite[Theorem~3]{Rosenthal} and obtain the result.
\end{proof}

The proof of Theorem~\ref{thm:besov} will be based on wavelet analysis, so let us first recall from \cite[Section~6.10]{Meyer} some basic facts on the existence of wavelet bases and how Besov spaces can be characterised by them.

For a given integer $R \ge 1$ there exist functions $\phi, \psi_{\lambda} \in L^2(\reals^d)$ ($\lambda \in \Lambda$) such that the following properties hold:
\begin{enumerate}[label={\rm (\roman*)}]
\item The functions $(\phi(x - k))_{k \in \integers^d}$ together with the functions $\psi_\lambda$ form an orthonormal basis of $L^2(\reals^d)$.
\item The index set $\Lambda$ equals $\bigcup_{j \ge 0} 2^{-j-1}\integers^d \setminus \{0\}$, and it is the disjoint union of the sets $\Lambda_j$ where $\Lambda_0 = \frac{1}{2}\integers^d$ and $\Lambda_j$ for $j \ge 1$ equals $2^{-j-1} \integers^d \setminus \bigcup_{k=0}^{j-1} \Lambda_k$.
\item There exist basic "mother wavelets" $\psi^{\nu}$ indexed by
$\nu \in \{0,1\}^d \setminus (0,\dots,0)$ such that $\psi_\lambda$ for
 $\lambda = 2^{-j}k + 2^{-j-1}\nu\in \Lambda_j$  is given by $\psi_\lambda(x) = 2^{jd/2} \psi^{(\nu)}(2^jx - k)$ , where $j$ is the level of the wavelet, $k$ is the shift and $\nu \in \{0,1\}^{d} \setminus (0,\dots,0)$ is the index of the basic mother wavelet. Notice that  the number of mother wavelets is  $2^d - 1$ in dimension $d$.
\item The functions $\phi$ and $\psi_\lambda$ are $R$ times differentiable.
\item There exists a compact set $K$ such that $\supp \phi, \supp \psi^{(\nu)} \subset K$ for all $\nu \in \{0,1\}^d$.
\end{enumerate}
Assume that $1 \le p,q \le \infty$ and $|s| < R$. If $f$ is a distribution with the wavelet series
  \[f(x) = \sum_{k \in \integers^d} \beta(k) \phi(x - k) + \sum_{j=0}^\infty \sum_{\lambda \in \Lambda_j} \alpha(\lambda) \psi_\lambda(x),\]
 then $f$ belongs to the Besov space $B_{p,q}^s(\reals^d)$ if and only if $\beta(k) \in \ell^p(\integers)$ and
  \begin{equation}\label{eq:besovcond}
    A_j \coloneqq 2^{dj(1/2 - 1/p)}2^{js}\Big(\sum_{\lambda \in \Lambda_j} |\alpha(\lambda)|^p \Big)^{1/p} \in \ell^q(\naturalnumbers).
  \end{equation}
We also recall that due to orthonormality of the wavelet basis, a coefficient in the wavelet representation of $f$ is obtained simply as the $L^2$-inner product (more precisely, as the  distributional duality) between  $f$ and  the corresponding wavelet. 

\begin{proof}[Proof of Theorem~\ref{thm:besov}]
  Choose a wavelet basis $\phi$, $\psi_\lambda$ with $R > d$. Then these wavelets can be used to study the local regularity of $\mu$, as we know it belongs to $H_{loc}^{-d-\varepsilon}(U)$. It is enough to prove the regularity of $\eta \mu$ in any small open ball $B \subset U$ of positive radius $r_0<r_d $ , where $r_d$ is as in  Lemma~\ref{le:exact} and where $\eta \in C_c^\infty(B)$. Namely, then  regularity in a given compact set  $K\subset U$ follows as we may cover $K$  by finitely many balls  of radius $r_0$ and  apply  a suitable smooth partition of unity. Moreover, by Lemma~\ref{lemma:reduction} it is enough to consider the exactly scaling field $X^e$.
  Let $\beta(k)$ and $\alpha(\lambda)$ be the coefficients of $\eta\mu^e$ in its wavelet series, formally
  \[\eta(x)\mu^e(x) = \sum_{k \in \integers^d} \beta(k) \phi(x - k) + \sum_{j=0}^\infty \sum_{\lambda \in \Lambda_j} \alpha(\lambda) \psi_\lambda(x).\]
  Because $\phi$ has a compact support and $\mu^e$ is $0$ outside of $B$, by the definition of the wavelet coefficients and the compact support of $\phi$   $\beta(k)$ is nonzero only for finitely many $k$. Hence
  \[\sum_{k \in \integers^d} |\beta(k)|^p < \infty\]
  almost surely. It is thus enough to determine when \eqref{eq:besovcond} holds.

  Assume that $r \ge 1$ is such that $\E |\mu^e_\beta(f)|^r \le C \|f\|_\infty^r$ for any $f \in C_c(B)$, where $C < \infty$ is some constant. Let $\psi_\lambda$ and $\lambda = 2^{-j} k + 2^{-j-1}\nu$ be as in (iii) above. Fix $j_0 \ge 0$ so that $2^{-j_0} K \subset B$, where $K$ is the support of $\psi^{(\nu)}$, and assume furthermore that $j_0$ is  large enough so that for all $j \ge j_0$ it holds that if $\supp \psi_\lambda \cap \supp \eta \neq \emptyset$, then $\supp \psi_\lambda \subset B$. Then by using the local translational invariance of the law of $\mu^e$ and the scaling relation in Lemma~\ref{le:exact} we have
  \begin{align*}
    \E |\alpha(\lambda)|^r & = \E \left| \int_B \eta(x) \psi_\lambda(x) \, d\mu^e(x) \right|^r \\
                           & = \E \left| \int_B \eta(x) 2^{jd/2} \psi^{(\nu)}(2^j x - k) \, d\mu^e(x) \right|^r \\
                           & = \E \left| \int_{2^{-(j-j_0)}B} \eta(x + 2^{-j} k) 2^{jd/2} \psi^{(\nu)}(2^j x) \, d\mu^e(x) \right|^r \\
                           & = \E \left| 2^{-(j-j_0)d} e^{\beta \RV - \frac{\beta^2}{2} (j-j_0) \log(2)} \int_B \eta(2^{-(j-j_0)}x + 2^{-j} k) 2^{jd/2} \psi^{(\nu)}(2^{j_0} x) \, d\mu^e(x) \right|^r \\
                           & = 2^{-(j/2-j_0)dr} \E e^{r \Re(\beta) \RV - \frac{\Re(\beta)^2 - \Im(\beta)^2}{2} (j-j_0) \log(2) r} \\
    & \quad \times \E \left| \int_B \eta(2^{-(j-j_0)}x + 2^{-j} k) \psi^{(\nu)}(2^{j_0} x) \, d\mu^e(x) \right|^r \\
                           & = 2^{-(j/2-j_0)dr} 2^{\frac{r^2 \Re(\beta)^2}{2} (j-j_0) - \frac{\Re(\beta)^2 - \Im(\beta)^2}{2} (j-j_0) r} \\
    & \quad \times \E \left| \int_B \eta(2^{-(j-j_0)}x + 2^{-j} k) \psi^{(\nu)}(2^{j_0} x) \, d\mu^e(x) \right|^r \\
    & \le C 2^{j\big(\frac{r(r-1)\Re(\beta)^2}{2} + \frac{r \Im(\beta)^2}{2} - \frac{dr}{2}\big)} \|\eta\|_\infty^r \|\psi^{(\nu)}\|_\infty^r.
  \end{align*}
  Above $\RV$ was as in Lemma~\ref{le:exact}, i.e. $\RV \sim N(0, (j-j_0) \log(2))$.
  By subadditivity we have for $r \le p$ and $j \ge j_0$ that
  \begin{align*}
    \E A_j^r & \le 2^{djr(1/2 - 1/p)}2^{rjs} \sum_{\lambda \in \Lambda_j} \E |\alpha(\lambda)|^r \\
             & \le C 2^{djr(1/2 - 1/p)}2^{rjs} 2^{jd} 2^{j\big(\frac{r(r-1) \Re(\beta)^2}{2} + \frac{r \Im(\beta)^2}{2} - \frac{dr}{2}\big)} \|\eta\|_\infty^r \|\psi^{(\nu)}\|_\infty^r \\
    & = C 2^{jr\big(-\frac{d}{p} + s + \frac{d}{r} + \frac{(r-1) \Re(\beta)^2}{2} + \frac{\Im(\beta)^2}{2}\big)} \|\eta\|_\infty^r\|\psi^{(\nu)}\|_\infty^r.
  \end{align*}
  Let
  \[\gamma \coloneqq -\frac{d}{p} + s + \frac{d}{r} + \frac{(r-1) \Re(\beta)^2 + \Im(\beta)^2}{2}.\]
  If we have $\gamma < 0$, then it follows from Chebyshev's inequality and Borel--Cantelli lemma that $A_j \in \ell^q$ almost surely for any $q \ge 1$. One easily checks that the minimum of $\gamma$ with respect to $r$ is attained at $r = \frac{\sqrt{2d}}{|\Re(\beta)|}$, so we should choose
  \[r = \min\big(p, \frac{\sqrt{2d}}{|\Re(\beta)|}\big).\]
  To see that this choice is valid, we have to check the moment condition in Theorem~\ref{thm:moments}. We may assume that $\Re(\beta), \Im(\beta) \ge 0$ as the other cases are symmetric. First consider the case $\Re(\beta) < \sqrt{\frac{d}{2}}$. Then it is enough to check what happens when $\Im(\beta)^2 = d - \Re(\beta)^2$. In this case it is enough to check the non-strict inequality and the condition becomes
  \[(r-1)\Re(\beta)^2 + d - \Re(\beta)^2 - \frac{2d(r-1)}{r} \le 0,\]
  which one easily checks holds when
  \[2 \le r \le \frac{d}{\Re(\beta)^2},\]
  and in particular it holds when $r = \frac{\sqrt{2d}}{\Re(\beta)}$.
  In the second case $\sqrt{\frac{d}{2}} \le r < \sqrt{2d}$. Now the worst case is when $\Im(\beta) = \sqrt{2d} - \Re(\beta)$ and the condition becomes
  \[(r-1)\Re(\beta)^2 + (\sqrt{2d} - \Re(\beta))^2 - \frac{2d(r-1)}{r} \le 0,\]
  which is satisfied if and only if $r = \frac{\sqrt{2d}}{\Re(\beta)}$.

  The proof is finished by plugging $r = \min(p, \sqrt{2d}/{|\Re(\beta)|})$ into the formula for $\gamma$ and solving $\gamma < 0$ for $s$.
\end{proof}

\end{document}